\documentclass{amsart}
\usepackage{amssymb,amsrefs}
\usepackage[colorlinks]{hyperref}
\usepackage{booktabs}
\usepackage[all]{xy}

\theoremstyle{plain}
\newtheorem{theorem}{Theorem}[section]
\newtheorem{proposition}[theorem]{Proposition}
\newtheorem{lemma}[theorem]{Lemma}

\theoremstyle{definition}
\newtheorem{remark}[theorem]{Remark}

\DeclareMathOperator{\Aut}{Aut}

\DeclareMathOperator{\disc}{disc}

\DeclareMathOperator{\End}{End}

\DeclareMathOperator{\GL}{GL}
\DeclareMathOperator{\SL}{SL}
\DeclareMathOperator{\Gal}{Gal}

\DeclareMathOperator{\ord}{ord}

\newcommand{\Ad}{{\mathbf{A}}}
\newcommand{\CC}{{\mathbf{C}}}
\newcommand{\FF}{{\mathbf{F}}}

\newcommand{\RR}{{\mathbf{R}}}
\newcommand{\Que}{{\mathbf{Q}}}
\newcommand{\Zee}{{\mathbf{Z}}}
\newcommand{\Zhat}{{\widehat {\Zee}}}

\newcommand{\Qbar}{{\bar{\Que}}}
\newcommand{\Ocal}{{\mathcal{O}}}

\newcommand{\Gotha}{{\mathfrak{A}}}
                                            
\newcommand{\gothp}{{\mathfrak{p}}}

\newcommand{\tto}{\longrightarrow}
\newcommand{\equi}{\Longleftrightarrow}

\newcommand{\iso}{\cong}

\newcommand{\mapright}[1]{\mathop{\longrightarrow}\limits^{#1}}

\newcommand{\ab}{\textup{ab}}

\newcommand{\mx}{\textup{max}}
\newcommand{\mn}{\textup{min}}
\newcommand{\tor}{\textup{tor}}
\newcommand{\cc}{\textup{cc}}
\newcommand{\cyc}{\textup{cyc}}

\newcommand\lowtilde{\lower0.7ex\hbox{\textasciitilde}}

\newcommand{\mybar}[1]{
  \mathchoice
  {#1\llap{$\overline{\phantom{\displaystyle\rm#1}}$}}
  {#1\llap{$\overline{\phantom{\textstyle\rm#1}}$}}
  {#1\llap{$\overline{\phantom{\scriptstyle\rm#1}}$}}
  {#1\llap{$\overline{\phantom{\scriptscriptstyle\rm#1}}$}}
}  
\renewcommand{\bar}{\mybar}

\begin{document}

\title[Adelic point groups of elliptic curves]
{Adelic point groups of elliptic curves}

\author[Angelakis]{Athanasios Angelakis} 

\author[Stevenhagen]{\hbox{Peter Stevenhagen}}

\address{Department of Mathematics,
         National Technical University of Athens, 
         9 Iroon Polytexneiou str.,
         15780 Zografou, 
         Attiki, Greece}

\address{Mathematisch Instituut,
         Universiteit Leiden, 
         Postbus 9512, 
         2300 RA Leiden, The Netherlands}
         
\email{ath.angelakis@gmail.com}
\email{psh@math.leidenuniv.nl}

\date{\today}


\begin{abstract}
We show that for an elliptic curve $E$ defined over a number field~$K$,
the group $E(\Ad_K)$ of points of $E$ over the adele ring 
$\Ad_K$ of $K$ is a topological group that can be analyzed
in terms of the Galois representation associated to the torsion points of $E$.
An explicit description of $E(\Ad_K)$ is given, and we prove that
for $K$ of degree $n$, `almost all' elliptic curves over $K$ have
an adelic point group topologically isomorphic to
$$
(\RR/\Zee)^n\times\Zhat^n\times\prod_{m=1}^\infty \Zee/m\Zee.
$$
We also show that there exist infinitely many elliptic curves over $K$
having a different adelic point group.
\end{abstract}

\maketitle

\section{Introduction}
\label{S:intro}

\noindent
Since the early 20th century, it has been a standard technique
to study number fields~$K$ in terms of their completions $K_\gothp$
at primes $\gothp$, both finite and infinite.
In the 1940s, all these completions were combined by Chevalley
in the adele ring $\Ad_K$ of~$K$.
This is a restricted direct product of completions, with
the integrality restriction in place to make $\Ad_K$ locally compact,
a property that all completions $K_\gothp$ have, and that is
essential in harmonic analysis.
The adele ring and its unit group, the idele group, play
an essential role in Tate's derivation of the functional equations
of Hecke $L$-functions and the idelic formulation of class field theory
\cite{CasselsFroehlich}.

It is a natural question whether $\Ad_K$ characterizes the number field~$K$,
i.e., whether non-isomorphic number fields can have topologically
isomorphic adele rings. 
Given the direct relation of $\Ad_K$ to basic invariants of
the number field such as the zeta function $\zeta_K$ of $K$ and
and the absolute abelian Galois group $G_K^\ab$ of~$K$,
this question neatly fits in a series of identical questions 
for $\zeta_K$, for $G_K^\ab$ and for the absolute Galois group $G_K$ itself.
It turns out that, whereas the topological group $G_K$ does 
characterize $K$ up to isomorphism by the 
Neukirch-Uchida theorem \cite{NeukirchSchmidtWingberg}*{12.2.1},
its maximal abelian quotient $G_K^\ab$, the adele ring $\Ad_K$
and the zeta function $\zeta_K$ do not; see \cite{Ang}*{Section 1.4}
for an extensive discussion and bibliography.

Non-isomorphic number fields with identical zeta-functions 
or isomorphic adele rings are very rare, and they do not exist
for small degrees.
Finding non-isomorphic number fields with isomorphic
absolute abelian Galois groups is even harder, and so far it has only been
achieved for imaginary quadratic fields.
Somewhat surprisingly, even though there are infinitely
many isomorphism types of $G_K^\ab$ for imaginary
quadratic $K$, we know that,
at least conjecturally \cite{AngSt}*{Conjecture 7.1},
many of them share the same `minimal' isomorphism type.

For elliptic curves $E$ defined over a number field $K$,
it is also standard to view them over the completions $K_\gothp$
to study their reduction properties, but less so over the adele ring $\Ad_K$,
maybe because it is not a field.
Still, it is a perfectly natural question what the
\emph{adelic point group} $E(\Ad_K)$ of $E$ over $\Ad_K$
looks like, and to which extent it characterizes $E/K$.
In view of Lemma~\ref{lemma:fullprod}, we can define it as the product
\begin{equation}
\label{eq:apg}
E(\Ad_K) = \prod_{\gothp\le\infty} E(K_\gothp)
\end{equation}
of the point groups of $E$ over all completions of $K$,
both finite and infinite.
This yields an uncountable abelian group that contains all
$\gothp$-adic point groups as subgroups, and continuously surjects onto the
point groups $\bar E(k_\gothp)$ at all primes of good reduction.
It is in a natural way a compact topological group.

In view of the fact that elliptic curves $E/K$ still give rise to basic
open questions such as the effective computation of $E(K)$,
the finiteness of its Tate-Shafarevich group, and the `average behavior'
of $E(K)$, it may come as a surprise that the adelic point group
$E(\Ad_K)$ not only admits a very explicit description
as a compact topological group, but that this description is also
almost universal in the sense that most $E$ over a given number field $K$
give rise to the \emph{same} adelic point group.

\begin{theorem}
\label{thm:almostall}
Let $K$ be a number field of degree $n$.
Then for almost all elliptic curves $E/K$, the adelic point group
$E(\Ad_K)$ is topologically isomorphic to the universal group
$$
  \mathcal{E}_n= (\RR/\Zee)^n\times \Zhat^n\times \prod_{m=1}^\infty \Zee/m\Zee
$$
associated to the degree $n$ of $K$.
\end{theorem}

\noindent
The notion of `almost all' in Theorem \ref{thm:almostall} is the same as in
\cite{Zywina}, and is based on the
counting of elliptic curves over $K$ given by short affine Weierstrass models 
\begin{equation}
\label{eq:E_a,b}
E_{a,b}:
y^2= x^3+ax+b
\end{equation}
with integral coefficients $a, b\in \Ocal_K$
satisfying $\Delta_E=\Delta(a,b)=-16 (4a^3+27b^2)\ne0$.
To define it, we fix a norm $||.||$ on the real vector space
$\RR\otimes_\Zee\Ocal_K^2\iso \RR^{2n}$ in which
$\Ocal_K^2$ embeds as a lattice.
Then for any positive real number $X$, the set $S_X$ of elliptic curves
$E_{a,b}$ with $||(a,b)||<X$ is finite, and we
say that \emph{almost all} elliptic curves over $K$ have 
some property if the fraction of elliptic curves $E_{a,b}$ in $S_X$
having that property tends to 1 when $X\in\RR_{>0}$ tends to infinity.

As the elliptic curves $E/K$ having complex multiplication (CM)
over $\overline K$ have their $j$-invariants inside a finite 
subset of $K$, almost all elliptic curves over $K$ are without CM.
This allows us to disregard CM-curves in the proof of Theorem
\ref{thm:almostall}, but we will see in Remark~\ref{rem:cmcase}
that this distinction is hardly relevant.

\medskip\noindent
Our proof of Theorem \ref{thm:almostall} is in 3 steps.
The first, in Section~\ref{S:adelicpointgroups},
only uses the standard theory of elliptic curves from \cite{SJH}.
It shows that the
connected component $E_{\cc}(\Ad_K)$ of the zero
element is a subgroup of $E(\Ad_K)$ isomorphic to $(\RR/\Zee)^n$,
and that it splits off in the sense that we have a decomposition
$$
E(\Ad_K) \iso E_{\cc}(\Ad_K) \times E(\Ad_K)/E_{\cc}(\Ad_K).
$$
The totally disconnected group $E(\Ad_K)/E_{\cc}(\Ad_K)$ is profinite,
and can be analyzed by methods resembling those we employed for 
the multiplicative group $\Ad_K^*$ in the class field theoretic 
setting of \cite{AngSt}.
It fits in a split exact sequence
$$
0 \to T_{E/K} \to E(\Ad_K)/E_{\cc}(\Ad_K) \to \Zhat^n \to 0
$$
of $\Zhat$-modules, with $n$ the degree of $K$.
Here $T_{E/K}$ is the closure of the torsion subgroup of
$E(\Ad_K)/E_{\cc}(\Ad_K)$, 
and we can write $E(\Ad_K)$ as a product
\begin{equation}
\label{eq:adelicdecomp}
E(\Ad_K) \iso (\RR/\Zee)^n\times \Zhat^n\times T_{E/K}
\end{equation}
in which only $T_{E/K}$ depends on the choice of the particular 
elliptic curve $E$ over~$K$.

In a second step, we show in Section \ref{S:torsion} that
the \emph{torsion closure} $T_{E/K}$, which is a countable product
of finite cyclic groups, is isomorphic to
$\prod_{m=1}^\infty \Zee/m\Zee$ for those $E$ that satisfy a
condition 
in terms of the division fields associated to $E/K$.
Whether this condition is satisfied can be read off from the Galois
representation associated to the torsion points of~$E$.
The final step concluding the proof of Theorem
\ref{thm:almostall}, in Section~\ref{S:universality},
uses recent results of Jones and Zywina \cites{Jones, Zywina}
to show that this condition is met for almost all elliptic curves over~$K$.

The notion of `almost all' from Theorem \ref{thm:almostall}
still allows for many $E/K$ having adelic point groups \emph{different} from 
the universal group~$\mathcal{E}_n$.
Such non-generic adelic point groups can be characterized by a finite set
of prime powers $\ell^k$ for which cyclic direct summands
of order $\ell^k$ are `missing' from $T_{E/K}$.
It is easy (Lemma \ref{lemma:elk=0}) to produce elliptic curves 
for which $E(\Ad_K)$ has prescribed missing summands by base changing 
any given elliptic curve to an appropriate extension field.
It is much harder to construct elliptic curves
with non-generic adelic point groups over a 
\emph{given} number field $K$.
Theorem \ref{thm:finite-ellk} shows that,
for given $K$, there are only finitely many prime powers $\ell^k$
for which cyclic direct summands of order~$\ell^k$ can be
missing from adelic point groups of elliptic curves $E/K$.
For $K=\Que$, the only prime power is $\ell^k=2$, and this
is used in Theorem \ref{thm:Efam} to prove an explicit version of the following
result.
\begin{theorem}
\label{thm:exceptions}
Let $K$ be a number field of degree $n$. Then there exist
infinitely many elliptic curves $E/K$ that are pairwise non-isomorphic
over an algebraic closure of~$K$, and for which $E(\Ad_K)$ is
a topological group not isomorphic to $\mathcal{E}_n$.
\end{theorem}

\section{Structure of adelic point groups}
\label{S:adelicpointgroups}

\noindent
Let $E$ be an elliptic curve over a number field $K$.
As $\Ad_K$ is a $K$-algebra inside the full product
$\prod_{\gothp\le\infty} K_\gothp$,
the adelic point group $E(\Ad_K)$ 
naturally embeds into $\prod_{\gothp\le\infty} E(K_\gothp)$.
The justification for our definition \eqref{eq:apg} is the following.
\begin{lemma}
\label{lemma:fullprod}
The natural inclusion $E(\Ad_K)\tto \prod_{\gothp\le\infty} E(K_\gothp)$
is an isomorphism.
\end{lemma}
\begin{proof}
The ring $\Ad_K$ consists of elements $(x_\gothp)_\gothp$
that are almost everywhere integral, i.e., 
for which we have $x_\gothp\in\Ocal_\gothp$ for almost all finite primes
$\gothp$  of $K$, with $\Ocal_\gothp\subset K_\gothp$ the local
ring of integers at $\gothp$.
For $E$ given by a projective model $E_{a,b}$ as in \eqref{eq:E_a,b},
every $K_\gothp$-valued point of $E$ with $\gothp$ finite
can be written with coordinates in~$\Ocal_\gothp$.
It follows that every element in $\prod_{\gothp\le\infty} E(K_\gothp)$
is actually in $E(\Ad_K)$.
\end{proof}

\noindent
As the structure of $E(K_\gothp)$ is different for archimedean
and non-archimedean $\gothp$, we treat these cases separately.
For archimedean primes, $K_\gothp$ is either $\RR$ or~$\CC$.
At complex places, $E(K_\gothp)$ is isomorphic to $(\RR/\Zee)^2$, 
as we have $E(\CC)\iso\CC/\Lambda$ for some lattice $\Lambda\subset \CC$.
At real places, the two possibilities for $E(K_\gothp)$ are
\[
E(K_\gothp) \iso
\begin{cases}
\mathbf{R}/\Zee,                    &\mathrm{if }\ \Delta_E <_\gothp 0;  \\
\mathbf{R}/\Zee \times \Zee/2\Zee,  &\mathrm{if }\ \Delta_E >_\gothp 0,  \\
\end{cases}
\]
depending on the sign of the discriminant $\Delta_E\in K^*$
under $\gothp:  K\to\RR$.

\begin{proposition}
\label{prop:archpart}
Let $K$ be a number field of degree $n$,
and $E/K$ an elliptic curve with discriminant $\Delta_E\in K^*/(K^*)^{12}$.
Then we have an isomorphism of topological groups
$$
\prod_{\gothp|\infty} E(K_\gothp) \iso (\RR/\Zee)^n \times (\Zee/2\Zee)^r.
$$
Here $r\le n$ is the number of real primes $\gothp$ of $K$ for which
we have $\Delta_E >_\gothp 0$.
\end{proposition}
\begin{proof}
Let $K$ have $r_1$ real and $r_2$ complex primes, then
$\prod_{\gothp|\infty} E(K_\gothp)$ is the product of $r_1+2r_2=n$
circle groups $\RR/\Zee$ and $r$ copies of $\Zee/2\Zee$,
with $r\le r_1\le n$.
\end{proof}

\noindent
To obtain the non-archimedean part $\prod_{\gothp<\infty} E(K_\gothp)$
of $E(\Ad_K)$, 
we take $\gothp$ finite and $E=E_{a,b}$ as in \eqref{eq:E_a,b}, and
consider the continuous reduction map
$
\phi_\gothp: E(K_\gothp) \to \bar{E}(k_\gothp)
$
to the finite set of points of the reduced curve $\bar{E}$
over $k_\gothp=\Ocal_K/\gothp$.
The set of points in the non-singular locus
$\bar{E}^{\mathrm{ns}}(k_\gothp)$ is contained in the image of $\phi_\gothp$ 
and inherits a natural group structure from $E(K_\gothp)$.
For $E_0(K_\gothp)=\phi^{-1}[\bar{E}^{\mathrm{ns}}(k_\gothp)]$
we obtain an exact sequence of topological groups
\begin{equation}
1\to E_1(K_\gothp)\tto E_0(K_\gothp) \tto
        \bar{E}^{\mathrm{ns}}(k_\gothp) \to 1,
\end{equation}
so the kernel of reduction $E_1(K_\gothp)$ is a subgroup of finite index
in $E_0(K_\gothp)$.
For primes of good reduction, the sequence simply reads
\begin{equation}
\label{seq:goodreducemodp}
1\to E_1(K_\gothp)\tto E(K_\gothp) \tto
        \bar{E}(k_\gothp) \to 1,
\end{equation}
and for $\gothp$ of bad reduction, 
$E_0(K_\gothp)\subsetneq E(K_\gothp)$ is of a subgroup of
\emph{finite} index \cite{SJH}*{VII.6.2}.
Either way, $E_1(K_\gothp)\subset E(K_\gothp)$ is a subgroup of finite index.

We can describe $E_1(K_\gothp)$ 
using the formal group of $E$ as in \cite{SJH}*{Chapter IV}.
More precisely, the elliptic logarithm
$
\log_\gothp: E_1(K_\gothp) \to \Ocal_\gothp
$
has a finite kernel of $p$-power order, and its image, which is
an open additive subgroup of the valuation ring $\Ocal_\gothp\subset K_\gothp$,
is non-canonically isomorphic to $\Zee_p^{[K_\gothp:\Que_p]}$.
As $E(K_\gothp)$ contains a
finitely generated $\Zee_p$-module
of free rank $[K_\gothp:\Que_p]$ as a subgroup of finite index,
its torsion subgroup $T_\gothp \subset E(K_\gothp)$ is finite,
and $E(K_\gothp)/T_\gothp$ a free $\Zee_p$-module of rank $[K_\gothp:\Que_p]$.
If we non-canonically write 
\begin{equation}
\label{eq:EK_p}
E(K_\gothp) \iso \Zee_p^{[K_\gothp:\Que_p]} \times T_\gothp
\end{equation}
and take the product over all non-archimedean primes $\gothp$ of $K$,
we can use the fact that the sum of the local degrees at the primes
over $p$ in $K$ equals $n=[K:\Que]$ to obtain the following non-archimedean 
analogue of Proposition \ref{prop:archpart}.
\begin{proposition}
\label{prop:nonarchpart}
Let $E$ be an elliptic curve over a number field $K$ of degree $n$.
Then we have an isomorphism of topological groups
$$
\prod_{\gothp<\infty} E(K_\gothp) =
\Zhat^n \times \prod_{\gothp<\infty} T_\gothp,
$$
where $T_\gothp = E(K_\gothp)^\tor$ is the finite torsion subgroup
of $E(K_\gothp)$.
\qed
\end{proposition}

\noindent
In order to combine Propositions \ref{prop:archpart} and 
\ref{prop:nonarchpart}, we define the profinite group
\begin{equation}
\label{eq:T_{E/K}}
T_{E/K}=\prod_{\gothp\le\infty}T_\gothp
\end{equation}
as a product of finite groups $T_\gothp$ given by
\begin{equation}
\label{eq:T_p}
T_\gothp =
\begin{cases}
E(K_\gothp)^\tor    &\mathrm{if}\ \gothp\ \mathrm{is\ finite;}  \\
\Zee/2\Zee          &\mathrm{if}\ \gothp\ \mathrm{is\ real\ and\ }\Delta_E >_\gothp 0;  \\
1    &\mathrm{otherwise.}  \\
\end{cases}
\end{equation}
With this definition, the adelic point group 
of $E$ over the number field $K$ of degree~$n$
is a topological group that can be written, as promised in 
\eqref{eq:adelicdecomp}, as
\[
E(\Ad_K)\iso (\RR/\Zee)^n\times \Zhat^n \times T_{E/K}.
\]
In this decomposition,
$(\RR/\Zee)^n$ is the connected component $E_{\cc}(\Ad_K)$ of 
the zero element in $E(\Ad_K)$,
and in the totally disconnected profinite group
$$
E(\Ad_K)/E_{\cc}(\Ad_K) \iso \Zhat^n \times T_{E/K},
$$
$T_{E/K}$ is the \emph{closure} of the torsion subgroup. 
Thus, the isomorphism type of $E(\Ad_K)$ 
is determined by the degree~$n$ of~$K$
and the structure of the \emph{torsion closure} $T_{E/K}$.

\section{Structure of the torsion closure}
\label{S:torsion}

\noindent
Let $T$ be any group that is obtained as a countable product
of finite abelian or, equivalently, finite cyclic groups.
Then there are usually many ways to represent $T$ as a product.
The group $\prod_{m=1}^\infty \Zee/m\Zee$ occurring in
Theorem \ref{thm:almostall}
is for instance isomorphic to $\prod_{m=2020}^\infty \Zee/m\Zee$, 
to $(\prod_{m=1}^\infty \Zee/m\Zee)^2$, and even to
$\prod_p \FF_p^*$.
Our choice is arbitrary, but requires only few characters to write it down.

In order to understand this notational ambiguity, and to deal with it,
we can represent a countable product $T$ of finite cyclic groups
in a more canonical way.
Writing each of the cyclic constituents as a product of
cyclic groups of prime power order and taking the cyclic groups
of each prime power order together, we arrive at its 
\emph{standard representation}
\[
T = \prod_{\ell \textrm{ prime}}\
\prod_{k=1}^\infty (\Zee/\ell^k\Zee)^{e(\ell^k)}.
\]
The exponents $e(\ell^k)$ are invariants of $T$, as they
can be defined intrinsically in terms of $T$ as
\[
e(\ell^k)=\dim_{{\mathbf F}_\ell}
T[\ell^k]/\left(T[\ell^{k-1}]+ \ell T[\ell^{k+1}]\right),
\]
We call $e(\ell^k)$ the \emph{$\ell^k$-rank} of $T$.
Clearly, two countable products of finite cyclic groups are
isomorphic if and only if their $\ell^k$-ranks coincide for
\emph{all} prime powers $\ell^k>1$.

The $\ell^k$-rank $e(\ell^k)$ of $T$ is either finite, in $\Zee_{\ge0}$,
or countably infinite.
In the latter case we write $e(\ell^k)=\omega$, and note that
we may identify
$$
(\Zee/\ell^k\Zee)^\omega=\hbox{Map}(\Zee_{>0}, \Zee/\ell^k\Zee)
$$
with the group of $\Zee/\ell^k\Zee$-valued functions
on the set~$\Zee_{>0}$ of positive integers.

The group 
$T_\mx=\prod_{m=1}^\infty \Zee/m\Zee$ occurring in Theorem \ref{thm:almostall}
has $e(\ell^k)=\omega$ for all prime powers $\ell^k$, as there are infinitely
many integers $m$ that are exactly divisible by $\ell^k$.
Leaving out finitely many $m$ from the product, or having each $m$ occur
twice, does not change the isomorphism type.
As there are infinitely many primes $p$ for which $p-1$ is exactly 
divisible by $\ell^k$, by the classical theorem of Dirichlet
on primes in arithmetic progressions, we also have 
$T_\mx\iso \prod_p \FF_p^*$, as claimed above.

In order to finish the proof of Theorem \ref{thm:almostall},
we need to show that for given $K$,
almost all $E/K$ have the property that
the torsion closure $T_{E/K}$ in \eqref{eq:T_{E/K}}
has infinite $\ell^k$-rank for \emph{all} prime powers~$\ell^k>1$,
making $T_{E/K}$ isomorphic to $T_\mx$.

In order to determine $e(\ell^k)$ for $T_{E/K}$,
we need to count how many
cyclic direct summands of order $\ell^k$ occur in 
$T_\gothp=E(K_\gothp)^\tor$ at finite primes $\gothp$ of $K$.
This can be done by studying the splitting behavior of $\gothp$
in the \emph{$\ell^k$-division fields}
\begin{equation}
Z_{E/K}(\ell^k) \stackrel{\textrm{def}}{=} K(E[\ell^k](\bar{K}))
\end{equation}
of $E$ over $K$.
The $\ell^k$-torsion subgroup
$E(K_\gothp)[\ell^k]\subset T_\gothp$
is full, i.e., isomorphic to $(\Zee/\ell^k\Zee)^2$,
if and only if $\gothp$ splits completely in 
$K\subset Z_{E/K}(\ell^k)$.

\begin{lemma}
\label{lemma:divfieldcriterion}
Let $E/K$ be an elliptic curve, and
$\ell^k>1$ a prime power for which the inclusion 
$$
Z_{E/K}(\ell^k)\subset Z_{E/K}(\ell^{k+1})
$$
of division fields is strict.
Then $T_{E/K}$ in \eqref{eq:T_{E/K}} has infinite $\ell^k$-rank.
\end{lemma}
\begin{proof}
Let $\gothp$ be a finite prime of $K$ that
splits completely in the division field $Z_{E/K}(\ell^k)$, but not in the 
division field $Z_{E/K}(\ell^{k+1})$.
Then $E(K_\gothp)$ has full $\ell^k$-torsion,
but not full $\ell^{k+1}$-torsion.
This implies that the finite group $T_\gothp$ contains a subgroup
isomorphic to $(\Zee/\ell^k\Zee)^2$ but not one isomorphic to
$(\Zee/\ell^{k+1}\Zee)^2$, and therefore has
at least one cyclic direct summand of order $\ell^k$.

By our assumption, the set of primes $\gothp$ that split completely in 
$Z_{E/K}(\ell^k)$, but not in $Z_{E/K}(\ell^{k+1})$, is infinite and
of positive density
$$
[Z_{E/K}(\ell^k):K]^{-1}-[Z_{E/K}(\ell^{k+1}):K]^{-1} >0.
$$
For all $\gothp$ in the infinite set 
thus obtained, the group $T_\gothp$
has a cyclic direct summand of order $\ell^k$.
It follows that $T_{E/K}=\prod_{\gothp\le\infty} T_\gothp$ has infinite
$\ell^k$-rank.
\end{proof}

\noindent
We conclude from Lemma \ref{lemma:divfieldcriterion}
that for an elliptic curve $E$ having the property
that for all primes $\ell$, the tower of $\ell$-power division fields 
\begin{equation}
\label{eqn:elltower}
Z_{E/K}(\ell) \subset Z_{E/K}(\ell^2) \subset Z_{E/K}(\ell^3)
\subset \cdots \subset Z_{E/K}(\ell^k) \subset \cdots
\end{equation}
has strict inclusions
at every level, $T_{E/K}$ is isomorphic to $\prod_{m=1}^\infty\Zee/m\Zee$.
In this situation, $E(\Ad_K)$ is isomorphic to the universal group
\begin{equation}
\label{eq:generic}
\mathcal{E}_n =
(\RR/\Zee)^n \times \Zhat^n \times \prod_{m=1}^\infty\Zee/m\Zee,
\end{equation}
in degree $n$,
and $E/K$ is \emph{generic} in the sense of Theorem \ref{thm:almostall}.

\section{Universality of $\mathcal E_n$}
\label{S:universality}

\noindent
In order to finish the proof of Theorem \ref{thm:almostall},
it suffices to show that for almost all $E$
defined over a fixed number field~$K$,
the extension $Z_{E/K}(\ell^k) \subset Z_{E/K}(\ell^{k+1})$
in the tower
\eqref{eqn:elltower} of $\ell$-power division fields is strict
for all prime powers $\ell^k>1$.
In order to see this for given $E/K$, we look at the Galois representation
\begin{equation}
\rho_{E/K}: G_K=\Gal(\bar K/K) \tto \Gotha_K=\Aut(E(\bar K)^\tor )
\end{equation}
on the group $E(\bar K)^\tor$ of its $\bar K$-valued torsion points.
The group $\Gotha_K$ is isomorphic to
$
\lim\limits_{\leftarrow m} \GL_2(\Zee/m\Zee)=\GL_2(\Zhat),
$
since $E(\bar K)^\tor$ is obtained as an injective limit
$$
E(\bar K)^\tor = 
\lim\limits_{\to m} E(\bar K) [m]\iso 
   \lim\limits_{\to m} (\frac{1}{m}\Zee/\Zee)^2 = (\Que/\Zee)^2.
$$
%
The restriction of the action of $G_K$ to the $m$-torsion subgroup
$E(\bar K) [m]$ of $E(\bar K)^\tor$ is described by
the reduction
$$
\rho_{E/K,m}: G_K\tto \Aut(E(\bar K)[m])\iso \GL_2(\Zee/m\Zee)
$$
of $\rho_{E/K}$ modulo $m$,
and the invariant field of $\ker\rho_{E/K,m}$ is 
the $m$-division field $Z_{E/K}(m)=K(E(\bar K)[m])$ of $E$ over $K$.
In particular, we have an equivalence
\begin{equation}
Z_{E/K}(\ell^k) = Z_{E/K}(\ell^{k+1}) \equi
     \ker[\rho_{E/K,\ell^k}] = \ker[\rho_{E/K,\ell^{k+1}}].
\end{equation}
In case $\rho_{E/K}$ is \emph{surjective}, all extensions 
$Z_{E/K}(\ell^k) \subset Z_{E/K}(\ell^{k+1})$ have degree 
\begin{equation}
\label{eqn:candegree}
\ell^4 =\#\ker [ \GL_2(\Zee/\ell^{k+1}\Zee) \tto \GL_2(\Zee/\ell^k\Zee) ],
\end{equation}
and in this case $E(\Ad_K)$ is isomorphic to the universal group
${\mathcal E}_n$ in \eqref{eq:generic}.

It is certainly not true that the \emph{image of Galois} $\rho_{E/K}[G_K]$
is always equal to the full automorphism group
$\Gotha_K=\Aut(E(\bar K)^\tor)$.
There is however the basic result, due to Serre \cite{Serre}, that 
$\rho_{E/K}[G_K]$ is an \emph{open} subgroup
of $\Gotha_K$ if $E$ is without CM, i.e., if $E$ does \emph{not} have
complex multiplication over $\bar K$.
In particular, the index of $\rho_{E/K}[G_K]$ in
$\Gotha_K\iso\GL_2(\Zhat)$ is always finite for $E$ without CM.
As observed in the introduction,
elliptic curves defined over $K$ with CM over $\bar K$
have their $j$-invariants in some \emph{finite} subset of $K$.
This is because for any number field $K$,
there are only finitely many imaginary quadratic orders $\Ocal$
for which the $j$-invariant $j(\Ocal)$ lies $K$.
As a consequence, almost all elliptic curves over $K$ are without CM.

We first look at the case $K=\Que$, which is somewhat particular as there
is a special complication that prevents $\rho_{E/K}$ in all cases
from being surjective.
In order to describe it, we let
$$
\chi_2: \Gotha_\Que=\Aut(E(\Qbar)^\tor)\tto \Aut (E[2](\Qbar))\iso
      \GL_2(\Zee/2\Zee)\iso S_3 \to \{\pm1\}
$$
be the non-trivial quadratic character
that maps an automorphism of $E(\Qbar)^\tor$ to the sign of the
permutation by which it acts on the
three non-trivial 2-torsion points of $E$.
For $\sigma\in G_\Que$, the sign of this permutation for 
$\rho_{E/\Que}(\sigma)$ is reflected in the action of $\sigma$ on the subfield
$\Que(\sqrt{\Delta})\subset Z_{E/K}(2)$ that is
generated by the square root of the
discriminant $\Delta=\Delta_E$ of the elliptic curve $E$,
and given by
$$
\chi_2(\rho_{E/\Que}(\sigma))=\sigma(\sqrt\Delta)/\sqrt\Delta.
$$
The Dirichlet character
$\Zhat^*\to \{\pm1\}$ corresponding to $\Que(\sqrt{\Delta})$
can be seen as a character
$$
\chi_\Delta: \Gotha_\Que\iso \GL_2(\Zhat)\mapright{\det}\Zhat^*\to \{\pm1\}
$$
on $\Gotha_\Que$.
It is different from the character $\chi_2$, which does not factor via
the determinant map $\Gotha_\Que\mapright{\det}\Zhat^*$ on $\Gotha_\Que$. 

The \emph{Serre character $\chi_E: \Gotha_\Que\to \{\pm1\}$} associated to $E$
is the non-trivial quadratic character
obtained as the product $\chi_E=\chi_2\chi_\Delta$.
By construction, it vanishes on the image of Galois
$\rho_{E/\Que}(G_\Que)\subset \Gotha_\Que$, so the
image of Galois is never the full group $\Gotha_\Que$.
In the case where we have $\rho_{E/\Que}(G_\Que)=\ker \chi_E $,
we say that $E$ is a \emph{Serre curve}.

If $E$ is a Serre curve, then the image of Galois is of index 2 in the full
group $\Gotha_\Que\iso \GL_2(\Zhat)$, and for every prime power $\ell^k>1$,
the extension
$$
Z_{E/K}(\ell^k)\subset Z_{E/K}(\ell^{k+1})
$$
of division fields for $E$ in Lemma \ref{lemma:divfieldcriterion}
has the degree~$\ell^4$ from \eqref{eqn:candegree}
for odd~$\ell$, and at least degree $\ell^3$ for $\ell=2$.
In particular, the hypothesis of Lemma \ref{lemma:divfieldcriterion} on $E$
is satisfied for all prime powers $\ell^k$ in case $E$ is a Serre curve.
Nathan Jones \cite{Jones} proved in 2010 that, in the sense of our Theorem
\ref{thm:almostall}, almost all elliptic curves are Serre curves.
This implies the case $K=\Que$ of Theorem \ref{thm:almostall} in~\cite{Ang}.
\begin{theorem}
\label{thm:almostallQ}
For almost all elliptic curves $E/\Que$, the adelic point group
$E(\mathbf{A}_\Que)$ is isomorphic to 
$
      \RR/\Zee\times\Zhat\times \prod_{m=1}^\infty \Zee/m\Zee.
$
\qed
\end{theorem}

\noindent
In order to deal with the case $K\ne\Que$, we need an
analogue of Jones' result stating that for almost all $E$ over $K$, the
image of Galois $\rho_{E/K}[G_K]\subset \Gotha_K$ is large.
As quadratic extensions of a number field $K\ne\Que$ are mostly
non-cyclotomic, there is no Serre character here.
However, for number fields $K$ that are not linearly disjoint from the 
maximal cyclotomic extension $\Que^\cyc=\Que(\zeta_\infty)$ of $\Que$,
the natural embedding
$$
\Gal(K^\cyc/K)\tto \Gal(\Que^\cyc/\Que)=\Zhat^*
$$
will not be an isomorphism, and identify $\Gal(K^\cyc/K)$ with some
open subgroup $H_K\subset \Zhat^*$ of index equal to the field degree
of the extension $\Que\subset (K\cap \Que^\cyc)$.
In this case, the image of Galois $\rho_{E/K}[G_K]\subset \Gotha_K$
is contained in the inverse image $\det^{-1}[H_K]$ of $H_K$
under the determinant map $\det: \Gotha_K\to\Zhat^*$.
We say that the image of Galois for an elliptic curve $E$
over $K\ne\Que$ is \emph{maximal} in case we have
\begin{equation}
\label{eqn:maximal}
\rho_{E/K}[G_K] = \det\nolimits^{-1}[H_K].
\end{equation}
For $K\ne\Que$, Zywina \cite{Zywina} proved in 2010 that,
in the sense of our Theorem \ref{thm:almostall},
almost all elliptic curves $E/K$ have maximal Galois image.
This allows us to finish the proof of our main theorem.

\medskip\noindent
\emph{Proof of Theorem \ref{thm:almostall}.}
The case $K=\Que$ is Theorem \ref{thm:almostallQ}.
For $K\ne\Que$, Zywina's result implies in particular that 
$$
\rho_{E/K}[G_K]=\det\nolimits^{-1}[H_K]\subset \Gotha_K
$$
contains $\ker[\det]\iso \SL_2(\Zhat)$ for almost all $E$.
It follows that for prime powers $\ell^k>1$, the degree of the 
extension $Z_{E/K}(\ell^k) \subset Z_{E/K}(\ell^{k+1})$ for these $E$
is maybe not
the maximal possible degree $\ell^4$ that we have in \eqref{eqn:candegree},
but it is still at least
\begin{equation}
\ell^3 =\#\ker [ \SL_2(\Zee/\ell^{k+1}\Zee) \tto \SL_2(\Zee/\ell^k\Zee) ].
\end{equation}
This implies that $T_{E/K}$ is the `maximal' group
$\prod_{m=1}^\infty \Zee/m\Zee$, and 
$E(\Ad_K)$ the generic group ${\mathcal E}_n$.
\qed

\begin{remark}
\label{rem:cmcase}
Even though, for the purpose of Theorem \ref{thm:almostall},
we can disregard all elliptic curves $E/K$ having CM,
one can show that also these curves typically have 
generic adelic point group, at least if we choose for $K$
the field of definition $\Que(j(E))$ of $E$.
This is because in the CM-case the relevant extension of division fields 
$
Z_{E/K}(\ell^k)\subset Z_{E/K}(\ell^{k+1})
$
from Lemma \ref{lemma:divfieldcriterion}
has generic degree $\ell^2$, and can be described explicitly in terms of
the ray class fields of conductor $\ell^k$ and $\ell^{k+1}$ associated
to the CM-order of $E$.
\end{remark}

\section{Non-generic point groups}
\label{S:non-generic}

\noindent
If the adelic point group $E(\Ad_K)$ is non-generic,
there is a prime power $\ell^k>1$ for which 
the inclusion of division fields
\begin{equation}
\label{eq:divfieldinclusion}
Z_{E/K}(\ell^k) \subset Z_{E/K}(\ell^{k+1})
\end{equation}
is an equality.
In case we can freely choose our ground field $K$,
it is easy to force equality in~\eqref{eq:divfieldinclusion} and
construct elliptic curves $E/K$ for which
the torsion closure $T_{E/K}$ has $\ell^k$-rank equal to 0
for \emph{any} prescribed finite set of prime powers $\ell^k$.
It suffices to take $m$ in the following Lemma divisible by
the appropriate powers $\ell^{k+1}$.

\begin{lemma}
\label{lemma:elk=0}
Let $E/\Que$ be any elliptic curve, $m\in\Zee_{>0}$ an integer,
and 
$$
K=Z_{E/\Que}(m)=\Que(E[m](\Qbar))
$$
the $m$-division field of $E$ over $\Que$.
Then $E$ is an elliptic curve defined over $K$, and
$T_{E/K}$ has $\ell^k$-rank $0$ for every prime power $\ell^k>1$ for which
$\ell^{k+1}$ divides $m$.
\end{lemma}
\begin{proof}
Suppose $\ell^{k+1}>\ell$ divides $m$. 
Then the full $\ell^{k+1}$-torsion
subgroup $E(\bar K)[\ell^{k+1}]$ is contained in $E(K)$,
so none of the torsion subgroups $T_\gothp$ of the
non-archimedean point groups $E(K_\gothp)$ in \eqref{eq:EK_p}
will have a a cyclic direct summand of order $\ell^k$.
As $K$ contains an $\ell^{k+1}$-th root of unity it has no real primes,
so by definition \eqref{eq:T_p} the torsion closure
$T_{E/K}=\prod_{\gothp\le\infty} T_\gothp$
has $\ell^k$-rank $0$.
\end{proof}

\noindent
In view of Lemma \ref{lemma:elk=0}, a more interesting question is 
which non-generic adelic point groups can 
occur over a \emph{given} number field $K$, such as $K=\Que$.
To realize non-generic adelic point groups, we need elliptic curves
$E/K$ and primes $\ell$ for which the tower of $\ell$-power division fields
\begin{equation}
\label{eqn:elltowerbis}
Z_{E/K}(\ell) \subset Z_{E/K}(\ell^2) \subset Z_{E/K}(\ell^3)
\subset \cdots \subset Z_{E/K}(\ell^k) \subset \cdots
\end{equation}
from \eqref{eqn:elltower} does \emph{not} have strict inclusions at every level.

To ease notation, we write $G_{\ell^k}=\Gal(Z_{E/K}(\ell^k)/K)$
for the Galois group over~$K$ of the $\ell^k$-division field, and 
$M_{\ell^k}=E[\ell^k](\bar K)$ for the group of $\ell^k$-torsion
points of~$E(\bar K)$.
As $M_{\ell^k}$ is free of rank 2 over $\Zee/\ell^k\Zee$ and
$G_{\ell^k}$ acts faithfully on $M_{\ell^k}$, we have an inclusion
\begin{equation}
\label{ell-k-action}
G_{\ell^k}\subset \Aut(M_{\ell^k})\iso\GL_2(\Zee/\ell^k\Zee),
\end{equation}
and we can view $\lim\limits_{\leftarrow k} G_{\ell^k}$ as a subgroup of 
$\Aut(\lim\limits_{\to k} M_{\ell^k})\iso\GL_2(\Zee_\ell)$.

The Galois group of the $(k-1)$-st extension in the tower 
\eqref{eqn:elltowerbis} is the $\ell$-group arising as the kernel
\begin{equation}
\label{eqn:elltowerkernels}
K_{\ell^k}=\ker [G_{\ell^k}\to G_{\ell^{k-1}}] \qquad (k\ge2)
\end{equation}
of the surjection induced by the restriction map
$\varphi_{\ell^k}: \Aut(M_{\ell^k})\to \Aut(M_{\ell^{k-1}})$,
so we have 
\begin{equation}
\label{eqn:elltowerkernelequiv}
K_{\ell^{k+1}}= 1 \Longleftrightarrow
Z_{E/K}(\ell^k) = Z_{E/K}(\ell^{k+1}).
\end{equation}
Triviality of the kernels $K_{\ell^{k+1}}$ needed to obtain
non-generic point groups can
only arise for a finite number of initial values of $k\ge1$,
with $\ell=2$ playing a special role.
\begin{proposition}
\label{prop:towerkernels}
Let $\ell$ be an odd prime, and suppose 
$K_{\ell^N}\ne1$ for some $N\ge2$.
Then we have $K_{\ell^k}\ne1$ for all $k>N$.
For $\ell=2$, the same is true if we have $N\ge3$.
\end{proposition}
\begin{proof}
Write $\sigma_N\in K_{\ell^N}\setminus\{1\}$ for $N\ge2$ as
$\sigma_N=1+\ell^{N-1}x$ with $x\ne \ell y\in \End(M_{\ell^N})$.
If $\sigma_k\in G_{\ell^k}$ for $k>N$ maps to $\sigma_N$
under the restriction map $G_{\ell^k}\twoheadrightarrow G_{\ell^N}$, 
we have $\sigma_k=1+\ell^{N-1}x_k$ with
$x_k\ne \ell y_k\in \End(M_{\ell^k})$, and
\begin{equation}
\sigma_k^\ell =
1+ \ell^N x_k + \sum_{i=2}^\ell {\ell\choose i} \ell^{i(N-1)}x_k^i =
1+ \ell^N x_k \in \End(M_{\ell^k}).
\end{equation}
Indeed, the number of factors $\ell$ in the coefficient
${\ell\choose i} \ell^{i(N-1)}$ is for $i=2, 3, \ldots, \ell-1$ at least
$1+2(N-1)\ge N+1$ if we have $N\ge 2$,
and for $i=\ell$ in the final coefficient $\ell^{\ell(N-1)}$ it is
$\ell(N-1)\ge N+1$ if we either have $\ell\ge3, N\ge 2$ or
$\ell=2, N\ge 3$.
Assuming we are in this situation, we see that $\sigma_k^\ell$ is in
$K_{\ell^{N+1}}\setminus\{1\}$.
Repeating the argument, we find that
if $\sigma_k$ is raised $k-N$ times
to the power $\ell$, we end up with an element
$1+\ell^{k-1}x_k\in K_{\ell^k}\setminus\{1\}$, showing $K_{\ell^k}\ne 1$.
\end{proof}

\noindent
In view of Proposition \ref{prop:towerkernels}, it makes sense to focus 
on the kernels $K_{\ell^2}$ and $K_8$ in~\eqref{eqn:elltowerkernels}.
\begin{proposition}
\label{prop:oddprimepowerdivfields}
Suppose $K$ is a number field linearly disjoint
from the $\ell^2$-th cyclotomic field $\Que(\zeta_{\ell^2})$,
with $\ell$ an odd prime.
Then for all elliptic curves $E/K$,
the tower \eqref{eqn:elltowerbis} has strict inclusions at all levels.
\end{proposition}
\begin{proof}
By Proposition \ref{prop:towerkernels}, it suffices to show that
$K_{\ell^2}=\ker[\pi: G_{\ell^2}\to G_\ell]$
is non-trivial for all elliptic curves $E/K$.
By the hypothesis on $K$, the determinant map
$G_{\ell^2}\mapright{\det} (\Zee/\ell^2\Zee)^*$ is surjective.
As $\ell$ is odd, we can pick $c\in G_{\ell^2}$ such that $\det(c)$
generates $(\Zee/\ell^2\Zee)^*$.
Applying $\pi$, we find that $\det(\pi(c))$ generates 
$\FF_\ell^*=(\Zee/\ell\Zee)^*$.

Suppose that $K_{\ell^2}$ is trivial, making $\pi$ an isomorphism.
Then the order of $\pi(c)$ equals the order
of $c$, which is divisible by the order $\ell(\ell-1)$ of
$(\Zee/\ell^2\Zee)^*$.
Let $s\in G_\ell$ be a power of $\pi(c)$ of order $\ell$.
Then $s\in G_\ell\subset \Aut(M_\ell)\iso \GL_2(\FF_\ell)$,
when viewed as a $2\times 2$-matrix over the field $\FF_\ell$,
is a non-semisimple matrix with double eigenvalue~1.
As $\pi(c)$ centralizes this element, its eigenvalues 
cannot be distinct, and we find that
$\det(\pi(c))$ is a square in $\FF_\ell^*$. Contradiction.
(This neat argument is due to Hendrik Lenstra.)
\end{proof}

\noindent
We can now show that, in contrast to Lemma \ref{lemma:elk=0},
there are only few ways in which adelic point groups of $E/K$
can be non-generic if we fix the base field $K$.
\begin{theorem}
\label{thm:finite-ellk}
Let $K$ be a number field. Then there exists a finite set
$\Sigma_K$ of powers of primes $\ell \mid 2\cdot\disc(K)$
such that for every elliptic curve $E/K$ and
for every prime power $\ell^k\notin \Sigma_K$,
the closure of torsion $T_{E/K}\subset E(\Ad_K)$ has 
infinite $\ell^k$-rank.
\end{theorem}
\begin{proof}
Suppose there exists a prime power $\ell^k$ and
an elliptic curve $E/K$ for which $T_E$ does
not have infinite $\ell^k$-rank.
Then we have $K_{\ell^{k+1}}=1$ in \eqref{eqn:elltowerkernelequiv}
for the associated tower 
\eqref{eqn:elltowerbis}.
If $\ell$ is odd, $K$ is not linearly disjoint from $\Que(\zeta_{\ell^2})$
by Proposition \ref{prop:oddprimepowerdivfields},
so~$\ell$~divides $\disc(K)$.
This leaves us with finitely many possibilities for $\ell$.

If $\ell$ is odd,
we have $K_{\ell^N}=1$ for $2\le N \le k+1$ by Proposition 
\ref{prop:towerkernels}, hence 
\begin{equation}
\label{eqn:divfieldeq}
Z_{E/K}(\ell^{k+1}) = Z_{E/K}(\ell).
\end{equation}
As $Z_{E/K}(\ell^{k+1})$ contains a primitive $\ell^{k+1}$-st root of unity
and $Z_{E/K}(\ell)$ is of degree at most $\#\GL_2(\FF_\ell)<\ell^4$ over $K$,
we can effectively bound $k$, say by $3+\ord_\ell(n)$, for $K$ of degree $n$.
For $\ell^k=2^k>4$, the argument is similar, using 
$Z_{E/K}(2^{k+1}) = Z_{E/K}(4)$ instead of \eqref{eqn:divfieldeq}.
\end{proof}
The proof of Theorem \ref{thm:finite-ellk} is constructive and yields
a set $\Sigma_K$ of prime powers~$\ell^k$,
but it does not automatically yield the \emph{minimal} set $\Sigma^{\mn}_K$.
By Lemma \ref{lemma:elk=0},
the minimal set $\Sigma^{\mn}_K$ for $K$ can include any given set of prime powers 
if we take $K$ sufficiently large.
Finding $\Sigma^{\mn}_K$ for any given $K$ is a however a non-trivial matter.

For $K=\Que$, one can take $\Sigma_\Que$ containing only powers of $\ell=2$,
and simple Galois theory shows that $\Sigma_\Que=\{2, 4, 8\}$ 
is actually large enough: no equality
$$
Z_{E/\Que}(2^{k+1}) = Z_{E/\Que}(4)
$$
can hold for $k\ge4$, as $Z_{E/\Que}(2^{k+1})$ then contains a cyclic subfield
$\Que(\zeta_{32}^{\hphantom{-1}}+\zeta_{32}^{-1})$
of degree 8 over $\Que$, whereas
$G_4=\Gal(Z_{E/K}(4)/\Que)\subset \GL_2(\Zee/4\Zee)$ has no elements
of order divisible by 8.
It is relatively easy to show that $\Sigma_\Que$ does contain 2, as
there is the following classical construction of a family of elliptic 
curves $E/\Que$ for which we have $Z_{E/\Que}(2)=Z_{E/\Que}(4)$.
See also \cite{DLR}*{Theorem 1.7}.
\begin{proposition}
\label{prop:z2isz4}
For every $r\in\Que^*$, the elliptic curve 
$$
E_r:\quad y^2=x(x^2-2(1-4r^4)x + (1+4r^4)^2)
$$
has division fields $Z_{E_r/\Que}(2)=Z_{E_r/\Que}(4)=\Que(i)$.
Conversely, every elliptic curve $E/\Que$ with
$Z_{E/\Que}(2)=Z_{E/\Que}(4)=\Que(i)$ is $\Que$-isomorphic to $E_r$
for some $r\in\Que^*$.
\end{proposition}

\begin{proof}
Let $E/\Que$ be defined by a
Weierstrass equation $y^2=f(x)$, and suppose that we have
$Z_{E/\Que}(2)=Z_{E/\Que}(4)=\Que(i)$.
Then $f\in\Que[x]$ is a monic cubic polynomial with splitting field
$Z_{E/\Que}(2)=\Que(i)$, so $f$ has one rational root, and two complex
conjugate roots in $\Que(i)\setminus\Que$.
After translating $x$ over the rational root,
we may take 0 to be the rational root of $f$, leading to the model
\begin{equation}
\label{eqn:model}
E: f(x)=x(x-\alpha)(x-\bar\alpha)
\end{equation}
for some element $\alpha\in\Que(i)\setminus\Que$.
Note that each such $\alpha$ does define an elliptic curve over $\Que$,
and that the $\Que$-isomorphism class of $E$ depends on
$\alpha$ up to conjugation and up to multiplication by the square
of a non-zero rational number.

The equality $Z_{E/\Que}(4)=\Que(i)$ means that the 4-torsion of $E$
is defined over $\Que(i)$, or, equivalently, that the 2-torsion
subgroup $E[2](\Que(i))$ of $E$ is contained in $2\cdot E(\Que(i))$.
In terms of the complete 2-descent map
\cite{SJH}*{Proposition 1.4, p.~315} over $K=\Que(i)$, which embeds
$E(K)/2E(K)$ in a subgroup of $K^*/(K^*)^2\times K^*/(K^*)^2$,
the inclusion $E[2](\Que(i))\subset 2\cdot E(\Que(i))$ amounts to
the statement that all differences between the roots of $f$ are
squares in $\Que(i)$.
In other words, we have $Z_{E/\Que}(2)=Z_{E/\Que}(4)=\Que(i)$ if and only if
$\alpha$ and $\alpha-\bar\alpha$ are squares in $\Que(i)$.

Writing $\alpha=(a+bi)^2$ with $ab\ne0$, we can scale $a+bi$ inside
the $\Que$-isomorphism class of $E$ by an element of $\Que^*$,
and flip signs of $a$ and $b$.
Thus we may take $\alpha=(1+qi)^2$, with $q$ a positive rational number.
The fact that $\alpha-\bar\alpha=4qi=(q/2)(2+2i)^2$ is a square in
$\Que(i)$ implies that $q/2=r^2$ is the square of some $r\in\Que^*$.
Substituting $\alpha=(1+2ir^2)^2$ in the model
\eqref{eqn:model},
we find that $E$ is $\Que$-isomorphic to
\begin{equation}
\label{eqn:Er}
E_r:\quad y^2=x(x^2-2(1-4r^4)x + (1+4r^4)^2).
\end{equation}
As we have shown that $E_r$ does have
$Z_{E_r/\Que}(2)=Z_{E_r/\Que}(4)=\Que(i)$, this proves the theorem.
\end{proof}
As the $j$-invariant $j(E_r)$ of the elliptic curve
given by \eqref{eqn:Er} is a non-constant function of $r$, the
family $\{E_r/\Que\}_{r\in\Que^*}$ is non-isotrivial, and represents
infinitely many distinct isomorphism classes over $\bar\Que$.
This yields the following explicit version of Theorem \ref{thm:exceptions}.
\begin{theorem}
\label{thm:Efam}
Let $E_r$ be as above, and $K$ be a number field of degree $n$.
Then all elliptic curves $E_r/K$ with $r\in\Que^*$ have
adelic point groups $E_r(\Ad_K)$ that are not isomorphic to the
topological group~$\mathcal E_n$.
\end{theorem}
\begin{proof}
We show that for any $r \in\Que^*$,
the torsion closure $T_{E_r/K}=\prod_\gothp T_\gothp$ from \eqref{eq:T_{E/K}}
has 2-rank equal to 0.
As $T_{E_r/K}$ is intrinsically defined as the closure of the torsion subgroup
of $E_r(\Ad_K)/E_{r,\cc}(\Ad_K)$, this implies that $E_r(\Ad_K)$ is not
isomorphic to $\mathcal E_n$.

As $E_r$ has by construction a non-complete 2-torsion subgroup
$E(\RR)[2]=\langle(0,0)\rangle$ over the unique archimedean completion
$\RR$ of $\Que$, its discriminant $\Delta(E_r)$ is a negative rational number.
By definition \eqref{eq:T_p}, we therefore have
$T_\gothp=1$ for all infinite primes $\gothp$ of $K$.

For $\gothp$ a finite prime of $K$, there are two cases.
If $K_\gothp$ contains $i$, and therefore $Z_{E_r/K}(4)=K(i)$,
the complete 4-torsion of $E_r$ is $K_\gothp$-rational,
and $E_r(K_\gothp)$ has no direct summand of order 2.
In the other case, where $K_\gothp$ does not contain~$i$, we have
$E(K_\gothp)[2]=\langle(0,0)\rangle$.
As all 4-torsion points of $E_r$ are $K_\gothp(i)$-rational, we can pick 
a point $P\in E_r(K_\gothp(i))$ of order 4 for which $2P$ is a 2-torsion
point $T\in E_r(K_\gothp(i))$ that is \emph{not} $K_\gothp$-rational. 
Write $\sigma$ for the non-trivial automorphism of
$K_\gothp\subset K_\gothp(i)$.
Then the point $Q=P+P^\sigma\in E_r[4](K_\gothp(i))$
is $K_\gothp$-rational and satisfies $2Q=T+T^\sigma=(0,0)$.
It follows that $E(K_\gothp)[4]=\langle Q\rangle\iso\Zee/4\Zee$
has no direct summand of order 2, and the same is then true for
$E(K_\gothp)$.
This shows that no $T_\gothp$ has a direct summand of order 2,
and completes the proof.
\end{proof}
\begin{remark}
It follows from the tables of Rouse and Zureick-Brown \cite{RZB}
and Proposition 3.4 in \cite{DLR}
that for \emph{all} elliptic curves $E/\Que$, the inclusion
$
Z_{E/\Que}(4) \subset Z_{E/\Que}(8)
$
is strict, and therefore, by 
Proposition \ref{prop:towerkernels}, that for such $E$
we have $K_{2^{k+1}}\ne 1$
in \eqref{eqn:elltowerkernelequiv} for all $k\ge2$.
This implies that for $K=\Que$,
$$
\Sigma^{\mn}_\Que=\{2\}
$$
is the minimal set of prime powers in Theorem \ref{thm:finite-ellk}.
\end{remark}
\begin{remark}
Both in Lemma \ref{lemma:elk=0} and in Theorem \ref{thm:Efam},
the only value of the $\ell^k$-rank of $T_{E/K}$
different from the generic value $e(\ell^k)=\omega$ is $e(\ell^k)=0$.
This is no coincidence, as there is the purely algebraic
fact that if a group $G$ acts on a free module $M$ of rank 2 over
$\Zee/\ell^{k+1}\Zee$ in such a way that the module of invariants
$M^G$ has a direct summand of order $\ell^k$, then there exists
an element $g\in G$ such that $M^{\langle g\rangle}$ has
a cyclic direct summand of order $\ell^k$.
Thus, if $E/K$ is an elliptic curve for which
$T_\gothp\subset T_{E/K}$ has a cyclic direct summand of order 
$\ell^k$ for a \emph{single} finite prime~$\gothp$,
then we are in the situation above, with
$G\subset G_{\ell^{k+1}}$ the decomposition group of~$\gothp$
acting on $M=M_{\ell^{k+1}}$ as in \eqref{ell-k-action}.
It then follows that for the infinitely many primes~$\gothp$ of~$K$
that are unramified in $K\subset Z_{E/K}(\ell^{k+1})$
with Frobenius element in $G_{\ell^{k+1}}$
conjugate to the element $g\in G$ above,
$T_\gothp$ also has a cyclic direct summand of order~$\ell^k$,
leading to the implication
$$
e(\ell^k)\ne 0 \Longrightarrow e(\ell^k)=\omega
$$
for the $\ell^k$-ranks of $T_{E/K}$.
\end{remark}

\bigskip


\enddocument